\documentclass[a4paper, oneside, 11pt]{amsart}
\usepackage[utf8]{inputenc}
\usepackage{mathtools}
\usepackage{amsfonts}
\usepackage{amsmath}
\usepackage{amssymb}
\usepackage{amsthm}
\usepackage[a4paper]{geometry}
\usepackage{mathrsfs}
\usepackage[dvipsnames]{xcolor}
\usepackage{dsfont}
\usepackage{hyperref}
\usepackage{tikz-cd} 
\usepackage{standalone}
\usetikzlibrary{arrows.meta,shapes.misc}
\usepackage[nameinlink]{cleveref} 
\usepackage{appendix}
\usepackage{enumerate}

\theoremstyle{plain}
\newtheorem{theorem}{Theorem}
\newtheorem{proposition}[theorem]{Proposition}
\newtheorem{lemma}[theorem]{Lemma}
\newtheorem{corollary}[theorem]{Corollary}
\newtheorem{definition}[theorem]{Definition}

\theoremstyle{definition}

\numberwithin{theorem}{section}
\numberwithin{equation}{section} 

\newcommand{\abs}[1]{\left\lvert #1 \right\rvert}

\newcommand{\vertiii}[1]{{\left\vert\kern-0.25ex\left\vert\kern-0.25ex\left\vert #1 
    \right\vert\kern-0.25ex\right\vert\kern-0.25ex\right\vert}}

\newcommand{\Z}{\mathbb{Z}}

\newcommand{\N}{\mathbb{N}}

\newcommand{\calB}{\mathcal{B}}
\newcommand{\calC}{\mathcal{C}}

\newcommand{\calN}{\mathcal{N}}

\newcommand{\calP}{\mathcal{P}}

\newcommand{\bfP}{\mathbf{P}}
\newcommand{\bfQ}{\mathbf{Q}}

\newcommand{\NN}{\overline{\mathcal N_o}}

\newcommand{\bbP}{\mathbb{P}}
\newcommand{\bbQ}{\mathbb{Q}}

\newcommand{\cond}{\big\vert}

\title[Comparison inequalities for infection processes]{Comparison inequalities for discrete- and continuous-time infection processes}

\author{Benedikt Jahnel$^{1,2}$}
\address{$^{1}$Technische Universit\"at Braunschweig}
\email{benedikt.jahnel@tu-braunschweig.de}
\email{jonas.koeppl@tu-braunschweig.de}
\author{Jonas K\"oppl$^{1}$}
\address{$^{2}$Weierstrass Institute Berlin}
\email{peterisklavs.silins@wias-berlin.de}
\author{Pēteris Siliņš{$^2$}}

\date{\today}
\usepackage[style = numeric-comp, sorting=nyt, url = false, abbreviate=false, maxbibnames=9, sortcites=true, doi = true, backend = biber, giveninits = true, isbn=false]{biblatex}
\renewbibmacro{in:}{\ifentrytype{article}{}{\printtext{\bibstring{in}\intitlepunct}}}
\DeclareFieldFormat{journaltitle}{\mkbibemph{#1\isdot}}
\begin{document}

\begin{abstract}
We study comparison inequalities for contact-process-type infection dynamics on $\Z^d$ with general local transmission mechanisms. The processes considered include both discrete-time oriented-percolation models and continuous-time contact processes in which infection events may transmit to random, possibly correlated sets of neighbouring sites. Our main results apply local-to-global comparison criteria beyond stochastic domination: if one local infection rule is more likely than another to hit every non-empty test set of neighbours, then the corresponding process has larger survival probabilities at all times.
We compare classical independent-infection models with exchangeable fixed-budget, all-or-nothing, and burst-type infection mechanisms, deriving orderings for example of finite-time and infinite-time survival probabilities and stopping-time distributions. In continuous time, we further obtain consequences for critical survival thresholds of exchangeable infection laws. The results show that infection processes with identical or comparable mean-field infection intensity can nevertheless be rigorously ordered at the stochastic level once spatial geometry is taken into account. Additionally, it turns out that the classical continuous-time contact process dominates a variety of related continuous-time models, whereas its most natural discrete-time version, Bernoulli oriented percolation, is not dominant with respect to related discrete-time models. 
\end{abstract}

\maketitle

\section{Introduction}
The contact process is one of the fundamental interacting particle systems and has served for decades as a paradigmatic model for the spread of infections or more generally  information on networks. Introduced in~\cite{Harris1974}, it has become a central object in probability theory and statistical physics, both for its role as the canonical lattice model of oriented percolation, see~\cite{durrett1984oriented}, as well as its  phase-transition behaviour as a continuous-time interacting particle system, see for example~\cite{Liggett1985,Liggett1999}.

In this short note, we investigate the  question of how different infection mechanisms influence the global behaviour of an epidemic. While the classical continuous-time contact process assumes that infected individuals infect their neighbours independently at exponential times, many applications call for more general transmission mechanisms. Examples include synchronous infections, burst-like transmission events, threshold effects, or infections that simultaneously affect several neighbouring sites. Such models arise naturally in epidemiology, ecology, and network science, where infection events often occur in correlated groups rather than independently, and where directionality, heterogeneity and correlation structure are known to materially affect risk and spread beyond what mean-field descriptions predict~\cite{allard2023role}. Comparisons between qualitatively different transmission mechanisms, e.g.\ horizontal versus vertical transmission, have also been studied in spatial stochastic models~\cite{schinazi2000horizontal}. In particular, burst production has been proposed as a biologically realistic alternative to continuous viral release and shown to substantially alter the early dynamics of viral infections, see e.g.,~\cite{PearsonKrapivskyPerelson2011,williams2024reproduction}.

From a probabilistic perspective, comparing different infection mechanisms is considerably more subtle than comparing infection rates within a fixed model. Classical comparison techniques are largely based on stochastic domination. For the contact process, for example Harris' graphical representation provides a powerful coupling construction that immediately yields monotonicity with respect to the infection rate and the initial condition. However, these methods typically require that one process dominates another in a pointwise sense and therefore do not apply when the local infection mechanisms are qualitatively different. 

Interestingly, many of the infection processes considered in this work share the same mean-field description despite exhibiting markedly different microscopic transmission mechanisms. In particular, the evolution of the average infection density is often governed by the same logistic equation, implying identical epidemic thresholds and equilibrium densities at the mean-field level. 
Our results therefore demonstrate that models that are indistinguishable under mean-field approximations can nevertheless be rigorously compared on the stochastic level in the presence of geometry.

The present work develops and applies comparison inequalities for infection processes that go beyond stochastic domination. The idea that local, vertex-wise ordering conditions can imply global comparisons of connectivity, without requiring stochastic domination, has a substantial history in the percolation and spatial-epidemics literature. Building on the clutter-percolation theorem of~\cite{mcdiarmid1981general}, the work~\cite{Kuulasmaa1982} introduced the framework of locally dependent random graphs to compare the extinction behaviour of spatial general epidemics with different, possibly correlated, local transmission mechanisms, and used this comparison to bound the critical infection rate of a class of SIR-type spatial epidemic models. This program was continued for example in~\cite{grimmett1998dependent}. 

Our main reference, however, is the recent work~\cite{baumler2026localcriteriaglobalconnectivity}, which revisits this local-to-global philosophy for degree-constrained percolation models on $\Z^d$. Here, we extend this philosophy from static percolation models to dynamic infection processes, returning the comparison technique to its original epidemiological setting. We introduce a general framework encompassing contact-process-type dynamics in both continuous and discrete time and derive local comparison criteria that imply global comparison inequalities for the resulting infection processes beyond stochastic domination. A principal example is the comparison between the classical contact process, where each infection event targets a single neighbouring site, and models in which an infection event simultaneously infects a prescribed number of neighbours. Although these processes exhibit fundamentally different local dynamics and admit no natural monotone coupling, our approach nevertheless establishes rigorous comparisons between their infection probabilities.

Besides providing new comparison tools for interacting particle systems, both in discrete and continuous time, our results illustrate that local inequalities can remain effective well beyond the setting of independent or positively associated models. We hope that this framework will prove useful for studying more general epidemic processes with correlated transmission mechanisms and stimulate further interactions between percolation theory and stochastic interacting particle systems.

The manuscript is organised as follows. We present the discrete- and continuous-time settings and the associated local-to-global comparisons in the following Section~\ref{sec_set}, which also contains our main applications as mentioned above. All proofs are exhibited in Section~\ref{sec_proof}.

\section{Setting and main results}\label{sec_set}
We consider infection processes on $\Z^d$, with $d\ge 1$, equipped with a nearest-neighbour structure. More precisely, $u\in \Z ^d$ has nearest-neighbours $\mathcal N_u=\{v\in \Z^d\colon \Vert u-v\Vert_1=1\}$ and we consider local Markovian dynamics for the evolution of a population of vertices in $\Z^d$ that are either {\em infected} or {\em susceptible}, both in discrete- as well as in continuous time. We consider a variety of nearest-neighbour infection mechanisms for which we provide comparisons for survival probabilities of the infection. Let us start with the discrete-time setting.

\subsection{Local-to-global comparison in discrete time}
In discrete time we consider infection processes based on local lattice-translation invariant and time-homogeneous infection rules. For this, let $N(x,n)$ be the random set of {\em infected neighbours} of $x$, which can include $x$ itself, infected at time $n+1$ from the space-time position $(x,n)$. More precisely, we set $N(x,n)=x+N'(x,n)$, where $N'(x,n)$ is an i.i.d.~copy of $N\subseteq \NN=\mathcal N_o\cup\{o\}$ for every $(x,n)$. We then say that $y\in N(x,n)$ is {\em infected} by $x$ or alternatively that the directed (space-time) edge between $(x,n)$ and $(y,n+1)$ is {\em open}.

We want to compare probabilities of sets of prescribed space-time paths to be open, in other words, infections to be passed through prescribed paths. For this, consider the {\em set of finite nearest-neighbour paths} of length $n\ge 0$,
$$\Psi_n:=\big\{\bar x=(x_0,\dots,x_n)\colon x_0=o,\, x_{i+1}-x_i\in \NN\text{ with }0\le i< n\big\},$$
 that start in the origin and write $\Psi:=\bigcup_{n\ge 0}\Psi_n$. For $\bar x\in \Psi$, we use $\abs{\bar x}$ to denote the length of the path. Now, for $\Xi\subset \Psi$ consider the event 
$$\calB^{\Xi}=\{\exists \bar x \in \Xi\colon x_{i+1}\in N(x_i,i)\text{ for all } 0\le i<|\bar x|\}
$$
that the infection can pass through at least one path in $\Xi$, see Figure~\ref{fig:discrete} for an illustration.
Note that $\Psi$ is countable and hence there are no measurability issues.
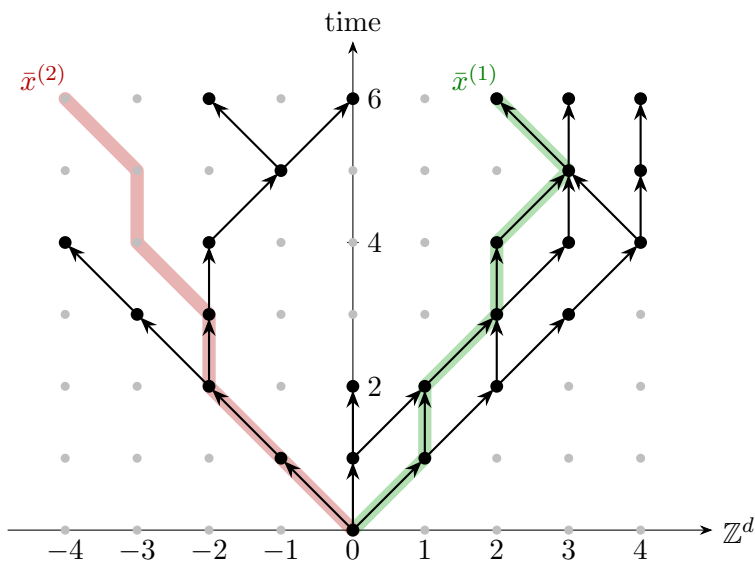
\begin{figure}[ht!]
\centering
\begin{tikzpicture}[>=Stealth, scale=0.95,
  site/.style={circle,fill=black!25,inner sep=1.2pt},
  inf/.style={circle,fill=black,inner sep=1.7pt},
  arr/.style={->,thick,shorten >=1.5pt},
  xione/.style={line width=5pt,green!60!black,opacity=.30,line cap=round,line join=round},
  xitwo/.style={line width=5pt,red!70!black,opacity=.30,line cap=round,line join=round}]
\draw[->] (0,0) -- (0,6.8) node[above] {time};
\foreach \n in {2,4,6} {\draw (-0.08,\n)--(0.08,\n); \node[left=4pt,fill=white,inner sep=1pt] at (0.6,\n) {$\n$};}
\draw[->] (-4.8,0) -- (5,0) node[right] {$\mathbb Z^d$};
\foreach \x in {-4,...,4} \node[below] at (\x,0) {$\x$};
\draw[xione] (0,0)--(1,1)--(1,2)--(2,3)--(2,4)--(3,5)--(2,6);
\draw[xitwo] (0,0)--(-1,1)--(-2,2)--(-2,3)--(-3,4)--(-3,5)--(-4,6);
\node[green!50!black] at (1.7,6.3) {$\bar x^{(1)}$};
\node[red!70!black]   at (-4.3,6.3) {$\bar x^{(2)}$};
\foreach \x in {-4,...,4} \foreach \n in {0,...,6} \node[site] at (\x,\n) {};
\foreach \x/\n/\y in {%
  0/0/-1, 0/0/0, 0/0/1,
  -1/1/-2, 0/1/0, 0/1/1, 1/1/1, 1/1/2,
  -2/2/-2, -2/2/-3, 1/2/2, 2/2/2, 2/2/3,
  -3/3/-4, -2/3/-2, 2/3/2, 2/3/3, 3/3/4,
  -2/4/-1, 2/4/3, 3/4/3, 4/4/3, 4/4/4,
  -1/5/-2, -1/5/0, 3/5/2, 3/5/3, 4/5/4}
  \draw[arr] (\x,\n) -- (\y,\n+1);
\foreach \x/\n in {%
  0/0,
  -1/1, 0/1, 1/1,
  -2/2, 0/2, 1/2, 2/2,
  -3/3, -2/3, 2/3, 3/3,
  -4/4, -2/4, 2/4, 3/4, 4/4,
  -1/5, 3/5, 4/5,
  -2/6, 0/6, 2/6, 3/6, 4/6}
  \node[inf] at (\x,\n) {};
\end{tikzpicture}
\caption{A graphical realisation of a discrete-time infection process with multiple outgoing infection arrows, alongside a set of paths $\Xi = \{\bar x^{(1)},\bar x^{(2)}\}$. The infection only passes through $\bar x^{(1)}$.}
\label{fig:discrete}
\end{figure}

Our first statement provides a comparison for the probabilities of events $\calB^{\Xi}$ for ordered distributions of the underlying local infection distributions. For this, let $\bfP,\bfQ$ be two probability distributions of $N$ and we write $\bbP$ and $\bbQ$, respectively, for the distribution of the associated space-time process.
\begin{theorem}[Local-to-global in discrete time]\label{thm_disc}
If $\bfP$ and $\bfQ$ are such that 
\begin{align*}
    \bfP[N\cap A\neq \emptyset]\le \bfQ[N\cap A\neq \emptyset],\qquad \text{ for all }\emptyset\subsetneq A\subseteq\NN,
\end{align*}
then, 
\begin{align*}
    \mathbb{P}[\calB^{\Xi}] \leq  \bbQ[\calB^{\Xi}], \qquad \text{ for all }\Xi\subset\Psi.
\end{align*}
\end{theorem} 
We present the proof and all other proofs in  Section~\ref{sec_proof}. In fact, a more general version of Theorem~\ref{thm_disc} is already presented in~\cite[Theorem 2.1]{Kuulasmaa1982} based on results in~\cite{mcdiarmid1981general}. There, the local comparison is phrased as an ordering of so-called {\em zero-functions} and the setting allows for more general graphs as well as infinite prescribed paths. However, the implied comparisons of degree-constrained percolation models, see below, are new. 

Let us highlight that the previous result is applicable also in the absence of stochastic domination of the local infection mechanisms. 
If $\Xi$ is the set of paths of length $n\in \N$, it provides for example a comparison for survival up to time $n$, and, by letting $n$ tend to infinity and continuity of measures, we also have domination of the global survival probabilities. Also, for $\Xi$ the set of paths of length $m\le n$ with $x_m=x$, the result provides a comparison for the probabilities of the events $\{\tau_x\le n\}$, where $\tau_x$ denotes the first hitting time of $x\in \Z^d$. This in turn would guarantee a nesting of the limiting deterministic shapes, once shape theorems for the models under consideration are established. This last task depends on the models at hand and we leave it to future work, not without mentioning that such shape theorems have been established for non-standard oriented percolation models for example in~\cite{cox1988limit}.

Theorem~\ref{thm_disc} becomes particularly interesting when we compare processes that are calibrated for example with respect to the expected number of outgoing infections per site. The following section is dedicated to comparisons in that spirit. 

\subsection{Applications in discrete time}\label{sec_application}
Recall that we assume that the process is started from a single infected vertex at time zero at the origin $o\in \Z^d$. In discrete time the dynamics can be seen as oriented percolation models where we highlight that additional edges are present connecting $(x,n)$ to $(x,n+1)$ for all $x\in \Z^d$ and $n\in \N_0$. Those extra edges can be treated without special attention bringing the models closer to corresponding general percolation models as presented in~\cite{baumler2026localcriteriaglobalconnectivity}.

However, in view of discrete-time contact processes it seems reasonable to interpret the additional edges as the potential healing of $x$ in the time interval $[n,n+1)$, which is often assumed to happen independently. In this section, we investigate both situations and start with a comparison similar to the one presented for degree-constrained percolation in~\cite{baumler2026localcriteriaglobalconnectivity}.

In order to perform calibrations, in this section, we denote by $\bfP_p$ the local infection distribution of $N$, wherein $x\in N$ with probability $p\in[0,1]$ independently for all $x\in \NN$. We refer to $p$ as the {\em infection probability} and let $\bbP_p$ denote the distribution of the associated space-time process. Note that this model closely resembles the standard $(d+1)$-dimensional oriented percolation model, with one extra edge in the time direction.

To derive other comparable  models, we associate with $p\in [0,1]$ two other parameters
\begin{align*}
    k:=\lfloor (2d+1)p\rfloor \in \{0,\ldots, 2d+1\} \quad \text{ and } \quad \varepsilon:= (2d+1)p - k\in [0,1),
\end{align*}
and define an infection model that we call the {\em $(2d+1)p$-nearest-neighbour infection process}. We write $\bfQ_p$ for the local distribution of $N\subset\NN$ defined as
\begin{align*}
    \bfQ_p[N = A] = \begin{cases}
        (1 - \varepsilon)\binom{2d+1}{k}^{-1} &\text{ if }\abs{A} = k,\\
        \varepsilon\binom{2d+1}{k+1}^{-1} &\text{ if }\abs{A} = k+1, k<2d+1\\
        0 & \text{ otherwise}.
    \end{cases}
\end{align*}
In words, $N$ consists of $k$ or $k+1$ uniformly chosen vertices in $\NN$ with probability $1-\varepsilon$ or $\varepsilon$, respectively.

As in ~\cite{baumler2026localcriteriaglobalconnectivity} we introduce the notion of exchangeable distributions on $\calP(\NN) = \{A\colon A\subseteq \NN\}$, which are convex combinations of $(2d+1)p$-nearest-neighbour infection processes defined as follows.

\begin{definition}[Exchangeability]
We say that a probability distribution $\bfP$ on $\calP(\NN)$ is {\em exchangeable} if, for all $k\in\{0,\ldots, 2d+1\}$ with $\bfP[\abs{N} = k]>0$, the conditional measure $\bfP[\cdot\mid \abs{N}=k]$ is a uniform distribution on $\{A\subseteq \NN\colon\abs{A}=k\}$. 
\end{definition}
As in~\cite{baumler2026localcriteriaglobalconnectivity} it holds that any such convex combination can be dominated from above and below by other local laws on $\calP(\NN)$ with the same expected degree. For this purpose we define the \textit{all-or-nothing} distribution $\bfP^{\text{aon}}_p$ via
\begin{align*}
    \bfP^{\text{aon}}_p[\abs{N} = 2d+1] = p, \quad \text{ and }\quad \bfP^{\text{aon}}_p[\abs{N} = 0] = 1 -p, \quad p\in [0,1].
\end{align*}
Then the all-or-nothing infection process is the least likely to survive indefinitely, while the $(2d+1)p$-nearest-neighbour infection process, where the infection is spread to $k$ or $k+1$ neighbours, is the most likely to survive indefinitely. Here and subsequently we write $\bfP[\abs{N}]$ for the expected size of $N$ under $\bfP$.
\begin{proposition}[Domination for exchangeable measures]\label{prop:exch}
    For all exchangeable probability measures $\bfP$ with $\bfP[\abs{N}] = (2d+1)p$, one has
    \begin{align*}
        \bfP^{\textup{aon}}_p [N\cap A\neq \emptyset]\leq \bfP[N\cap A \neq \emptyset]\leq \bfQ_p[N\cap A\neq \emptyset],\quad \text{ for all }\emptyset\subsetneq A\subseteq\NN.
    \end{align*}
\end{proposition}
In particular the classical oriented percolation model is sandwiched in the sense that for all $p\in [0,1]$ we have that
    \begin{align*}
       \bfP^{\textup{aon}}_p [N\cap A\neq \emptyset]\leq \bfP_p[N\cap A\neq \emptyset]\leq \bfQ_p[N\cap A\neq \emptyset],\quad \text{ for all }\emptyset\subsetneq A\subseteq\NN. 
    \end{align*}

As alluded to above, all of the aforementioned models can be modified to more closely resemble what might be imagined as a discrete-time contact process with independent recoveries. More precisely, we may adapt all of the previous models such that the origin is infected separately from the $2d$ nearest neighbours. Specifically, we may assume that for $A\subseteq \calN_o$, the events $\{N\setminus\{o\} = A\}$ and $\{o\in N\}$ are independent with $\{o\not\in N\}$ having probability $0\leq q\leq 1$. We call this the {\em recovery probability} and say that there was a recovery at $(x,n)\in\Z^d\times\N_0$ if $x\notin N(x,n)$. Such models we refer to as  having {\em independent recoveries}, and $q = 1$ implies that the additional edges are removed entirely.
The statement of Proposition~\ref{prop:exch} transfers to models with independent recoveries if we keep the recovery probability fixed in all models and compare only the infection laws for the remaining neighbours in $\mathcal N_o$. In this case the measures are calibrated to have mean $2dp + (1 - q)$. 

Finally, in view of applications where an infected cell transmits infection only upon death, as studied in a biological setting, see for example~\cite{PearsonKrapivskyPerelson2011}, we consider {\em burst infection processes} where at each space-time position recovery happens with positive probability $q$ and, if the vertex does not recover, it also does not spread the infection to any neighbour.

\begin{definition}[Burst infection]
    We say that a probability measure $\hat{\bfP}$ on $\calP(\NN)$ is a {\em burst probability measure} if $\hat{\bfP}[o\notin N]>0$ and $\hat{\bfP}[\{\abs{N}>1\}\cap\{o\in N\}] = 0$.
\end{definition}

There are some cases where the burst infection process can outlive a non-burst process, even if their expected infection intensity is the same. 
\begin{lemma}[Domination by burst-infection models]\label{lemma:burst}
    Let $\bfP$ be a measure on $\calP(\NN)$ that has independent recoveries with probability $q>0$ and for which $\bfP[\abs{N\setminus\{o\}}>0 ]\leq q$. Then, there exists a burst measure $\hat{\bfP}$ for which
    $\hat{\bfP}[o\notin N] =q$, $\bfP[\abs{N}] = \hat{\bfP}[\abs{N}]$, $\bfP[N\setminus \{o\} = A] = \hat{\bfP}[N\setminus \{o\} = A]$ for all $ A\subseteq \calN_o$, and
    \begin{align*}
        \bfP[N\cap A\neq \emptyset]\leq \hat{\bfP}[N\cap A\neq \emptyset],\qquad\text{for all  }\emptyset\subsetneq A\subseteq \NN.
    \end{align*}
\end{lemma}

Before we come to the time-continuous setting, we note that domination results in the spirit of our Proposition~\ref{prop:exch} have also been derived in~\cites[Theorem 4.1\addsemicolon]{Kuulasmaa1982}{Kuulasmaa_Zachary_1984} for a class of SIR models.
Percolation-based bounding techniques for spatial and network epidemics more broadly have also been developed in~\cite{meester2011bounding,miller2008bounding}, though without the local-to-global comparison structure used here.
\subsection{Local-to-global comparison in continuous time}
For the continuous-time setting we fix a collection $R=(R_x)_{x\in \Z^d}$ of i.i.d.~Poisson point processes with rate one representing the times at which the individual at vertex $x\in \Z^d$ recovers. Additionally, fix $\lambda \geq 0$ and let $I=(I_x)_{x \in \Z^d}$ be another i.i.d.~family of Poisson point processes with rate $2d\lambda$, independent from  $R$ that represents the times at which the vertex $x$ can possibly infect some of its nearest neighbours. 
Now, at every space-time position $(x,t)$, with $t\in I_x$, independently of everything else, a random set of neighbours $N(x,t)=x+N'(x,t)$ gets infected. Here, the $N'(x,t)$ are i.i.d.~copies of $N\subseteq\mathcal N_o$. 

In view of the discrete-time setting, for a realisation of $(I,R)$, consider the set of finite nearest-neighbour space-time paths of length $n\ge 0$,
\begin{align*}
\Psi_n=\Psi_n(I,R):=\big\{\bar x=\big((x_0,t_0),&\dots,(x_{n-1},t_{n-1}), x_n\big)\colon x_0=o\text{ and for all }0\le i< n\\
&x_{i+1}- x_i\in \mathcal N_o,\, t_i\in I_{x_i}\text{ with }t_0\le \dots \le t_{n-1}\text{ and}\\
&R_{x_0}\cap [0,t_0]=\emptyset,\, R_{x_i}\cap [t_{i-1},t_{i}]=\emptyset \text{ for all }0< i< n
\big\}
\end{align*}
that start in the origin and are potential infection paths towards a final vertex $x_n$. Note that $\Psi_n$ is almost-surely finite. Write $\Psi:=\bigcup_{n\ge 0}\Psi_n$ and denote for a measurable $\Xi=\Xi(I,R)\subset \Psi(I,R)$ by
$\mathcal B^\Xi$
the indicator of the event that there exists a path $\big((x_0,t_0),\dots,(x_{n-1},t_{n-1}), x_n\big)\in \Xi$ such that
$$
x_{i+1}\in N(x_i,t_i)\qquad\text{ for all }0\le i<n.
$$
In words, there is an infection path from $(o,0)$ inside $\Xi$, where $\Xi$ may additionally depend on the realisation of $(I, R)$, that completely avoids the recovery times, see Figure~\ref{fig:continuous} for an illustration. 

Let us highlight that, compared to the discrete-time setting, infections are only sent to the nearest neighbours, whereas the infected vertex itself remains infected until it recovers via its independent recovery process. Additionally, the infections are transmitted at the infection times and not ``diagonally" one time unit further as in the discrete-time setting. 
\begin{figure}[t]
\centering
\begin{tikzpicture}[>=Stealth, x=1.7cm, y=0.95cm,
  arr/.style={->,thick},
  rec/.style={draw,cross out,thick,minimum size=6pt,inner sep=0pt},
  ev/.style={circle,fill=black,inner sep=1.4pt},
  xione/.style={line width=6pt,green!60!black,opacity=.30,line cap=round,line join=round},
  xitwo/.style={line width=6pt,red!70!black,opacity=.30,line cap=round,line join=round}]
\foreach \x in {-2,-1,1,2} \draw[gray!60] (\x,0)--(\x,6.2);
\draw[->] (0,0) -- (0,6.6) node[above] {time $t$};
\draw[->] (-2.6,0) -- (2.75,0) node[right] {$\mathbb Z^d$};
\foreach \x in {-2,...,2} \node[below] at (\x,0) {$\x$};
\foreach \x/\a/\b in {0/0/3.4, 1/0.8/1.8, 1/2.6/4.5, 2/1.2/2.3, 2/3.1/5.2, -1/2.6/3.3, -2/2.9/3.3}
  \draw[line width=2.4pt] (\x,\a)--(\x,\b);
\draw[xione] (0,0)--(0,0.8)--(1,0.8)--(1,1.2)--(2,1.2)--(2,1.5);
\draw[xitwo] (0,0)--(0,2.6)--(-1,2.6)--(-1,2.9)--(-2,2.9)--(-2,4.0);
\draw[green!50!black,thick] (1.9,1.5)--(2.1,1.5) node[right] {$\bar x^{(1)}$};
\draw[red!70!black,thick]   (-2.1,4.0)--(-1.9,4.0) node[right] {$\bar x^{(2)}$};
\foreach \x/\t in {0/3.4, 0/5.8, 1/1.8, 1/4.5, 2/2.3, 2/5.2, -1/3.3, -1/4.9, -2/1.0, -2/3.3}
  \node[rec] at (\x,\t) {};
\foreach \x/\t/\y in {%
  0/0.8/1, 0/2.6/-1, 0/2.6/1,
  1/1.2/2, 1/3.1/0, 1/3.1/2,
  2/0.3/1, 2/4.0/1,
  -1/0.5/-2, -1/2.9/-2,
  -2/2.0/-1, -2/4.6/-1}
  {\node[ev] at (\x,\t) {}; \draw[arr] (\x,\t) -- (\y,\t);}
\end{tikzpicture}
\caption{A graphical realisation of a continuous-time infection process with a variable number of outgoing infection arrows, alongside a set of paths $\Xi(I,R) = \{\bar{x}^{(1)}, \bar{x}^{(2)}\}$. The infection only passes through $\bar x^{(1)}$. }
\label{fig:continuous}
\end{figure}
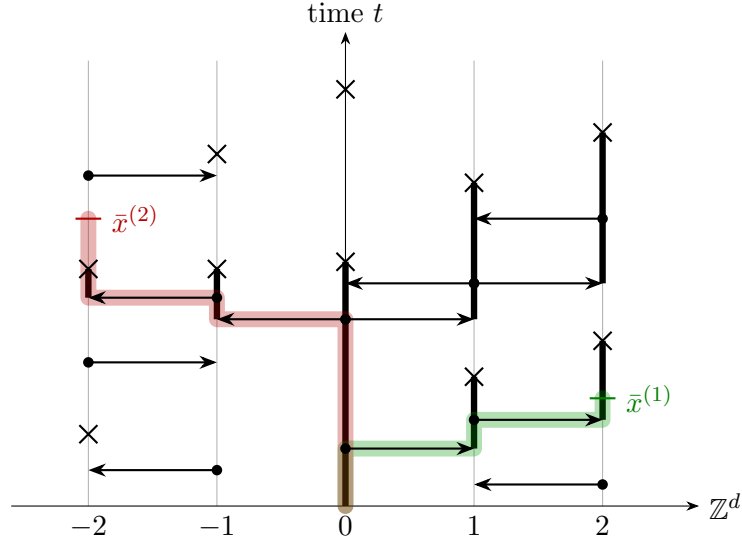

Our second main result now provides a criterion that allows to compare the probabilities of events $\mathcal B^\Xi$ based on a local comparison of the distributions of $N$. Again, we consider distributions $\bfP$ and $\bfQ$ of $N$, and denote by $\bbP$ and $\bbQ$ the distribution of the infinite space-time system based on the distribution of $(I,R)$ and the i.i.d.~copies $N'(x,t)$ based on $\bfP$ and $\bfQ$, respectively.

\begin{theorem}[Local-to-global in continuous time]\label{thm_cont}
If $\bfP$, $\bfQ$ are such that 
\begin{align*}
    \bfP[N\cap A\neq \emptyset]\le \bfQ[N\cap A\neq \emptyset],\qquad \text{ for all }\emptyset\subsetneq A\subseteq\mathcal N_o,
\end{align*}
then,
\begin{align*}
    \mathbb{P}[\calB^{\Xi}] \leq  \mathbb{Q}[\calB^{\Xi}], \qquad \text{ for all measurable }\Xi.
\end{align*}
\end{theorem}
Again, we highlight that the above result can be used to compare processes for which no monotone coupling is available. For $t\ge 0$ and the measurable set
\begin{align*}
\Xi(I,R)=\big\{&\big((x_0,t_0),\dots,(x_{n-1},t_{n-1}), x_n\big)\in \Psi\colon t_{n-1}\le t,\,  R_{x_n}\cap [t_{n-1},t] = \emptyset
\big\},
\end{align*}
we have that $\mathcal B^\Xi$ is the indicator of the event that the infection survives up until time $t\ge0$, and letting $t$ go to infinity, our theorem provides a criterion for the ordering of global survival. Similarly, for $t\ge 0$ and the measurable set
\begin{align*}
\Xi(I,R)=\big\{&\big((x_0,t_0),\dots, (x_{n-1},t_{n-1}),x_n\big)\in \Psi\colon x_n = x, t_{n-1}\leq t
\big\},
\end{align*}
i.e., the set of paths ending in $x\in \Z^d$ before time $t\ge 0$, the indicator $\mathcal B^\Xi={\bf 1}\{\tau_x\le t\}$, with $\tau_x$ the first hitting time of $x\in \Z^d$, can be used to order first-hitting probabilities. In particular, in case the existence of limiting shapes for rescaled infected sets of the comparable models can be guaranteed, our results can be used to determine a nesting of those sets. 

Next, we provide comparisons for a variety of infection distributions that are appropriately calibrated. 
\subsection{Applications in continuous time}
In the discrete-time setting we mostly made comparisons between measures on $\calP(\NN)$ for which the expected number of infected neighbours at each space-time point was the same. In the continuous-time setting we compare measures on $\calP(\calN_o)$ for which the expected number of outgoing infections in a unit time interval is calibrated to match $2d\lambda$, where we recall that $2d\lambda$ is the expected number of outgoing infections in the classical contact process based on $I$.  Our comparisons will always use the same underlying infection times $I$ and recovery times $R$ but different infection distributions at the infection times.

A few of the infection models can be directly taken from the discrete-time setting. In particular, any exchangeable measure over $\calP(\calN_o)$ with mean one is a valid candidate for a distribution of $N$. Some examples include the all-or-nothing measure $\bfP^{\text{aon}}$, where
\begin{align*}
    \bfP^{\text{aon}}[\abs{N} = 2d] =1/(2d)\quad\text{ and }\quad \bfP^{\text{aon}}[\abs{N} = 0] = 1 -1/(2d).
\end{align*}
 We might also consider the Bernoulli measure $\hat{\bfP},$ where $x\in N$ independently for all $x\in \calN_o$ and with probability $1/(2d)$. 

Specifically in the continuous-time case we define the following infection model. Consider a parameter $a\in [1,2d]$, to which we associate two other parameters 
\begin{align*}
    k:=\lfloor a\rfloor\in\{1,\ldots, 2d\}\quad\text{ and }\quad \varepsilon:= a - k\in [0,1).
\end{align*}
We define the measure $\hat{\bfQ}_a$ on $\mathcal P(\calN_o)$ such that $\hat{\bfQ}_a[\abs{N} = k] = (1-\varepsilon)/k, \hat{\bfQ}_a[\abs{N} = k+1] = \varepsilon/(k+1)$, and $\hat{\bfQ}_a[|N|\notin\{0,k,k+1\}] = 0$, i.e., the remaining mass concentrates on the empty set. Conditionally on $\abs{N}$ being $k$ or $k+1$, $N$ is a uniform subset of $\calN_o$ of size $k$ or $k+1$, respectively. We call $\hat{\bfQ}_a$ the {\em budget-$a$ measure}, and let $\hat{\bbQ}_a$ denote the corresponding distribution of the associated {\em budget-$a$ space-time process}. We note that, by definition we have the calibration $\hat{\bfQ}_a[\abs{N}] = 1$ for $a\in[1,2d]$. We now present the following result that gives an ordering of this family of measures.

\begin{proposition}[Ordering of budget-$a$ measures]\label{prop:cts_ord1}
    For $a,b\in[1,2d]$, $a<b$, let $\hat{\bfQ}_a, \hat{\bfQ}_b$ be defined as above. Then,
    \begin{align*}
        \hat{\bfQ}_b[N\cap A\neq\emptyset]\le \hat{\bfQ}_a[N\cap A\neq\emptyset], \qquad\text{ for all }\emptyset\subsetneq A\subseteq \calN_o.
    \end{align*}
\end{proposition}
Note that, with this notation, $\hat{\bfQ}_1$ is the local law of the classical contact process and, in view of the discussion following Proposition~\ref{prop:exch} above, it even dominates any other exchangeable measure $\hat{\bfQ}$ for which $\hat{\bfQ}[\abs{N}] = 1$. 

Let us highlight that the classical local infection law is dominant for the continuous-time contact process compared to other continuous-time relatives. This is contrasted by the lack of domination of its most natural discrete-time version, Bernoulli oriented percolation, with respect to other discrete-time processes, see Section~\ref{sec_application}.  

Let us finally rephrase our results with respect to the ordering of the corresponding critical thresholds. For this, let $\lambda_c$ be the critical global-survival threshold in the classical contact process, i.e., where $|N|=1$ almost surely, and
$\lambda^\bfP_c$ the critical parameter for global survival for any $\bfP$ on $\calP(\calN_o)$.
\begin{corollary}[Survival and extinction regimes for exchangeable measures]\label{cts_crit}
For any exchangeable $\bfP$ with $\bfP[\abs{N}] = 1$ and $\bfP[\abs{N}> 0]>0$, we have that
    \begin{align*}
        \lambda_c \leq \lambda^\bfP_c\leq \lambda_c/\bfP[\abs{N}>0].
    \end{align*}
\end{corollary}

In the next section, we provide all proofs, starting with the discrete-time setting.

\section{Proofs}\label{sec_proof}
\subsection{Discrete time}
We start with our first main result on a local-to-global comparison, which is a version of the corresponding result for percolation in~\cite{baumler2026localcriteriaglobalconnectivity} adapted to the oriented setting.
\begin{proof}[Proof of Theorem~\ref{thm_disc}]

Assume that $\Xi$ is finite. Then, the set of all space-time points reached by some path in $\Xi$, i.e., 
\begin{align*}
    \calC=\calC(\Xi):=\bigcup_{\bar x\in \Xi}\bigcup_{0\le i\le |\bar x|-1}\big\{(x_i,i)\in\bar{x}\big\},
\end{align*}
is also finite. Now, for a subset $U\subseteq \calC$, we define the measure $\bbQ_U$ by independently sampling $N'(x,i)\sim \bfQ$ for all $(x,i)\in U$ and $N'(y,j)\sim \bfP$ for all $(y,j)\in \calC\setminus U$. Let $U\subseteq \calC$ and $(x,m)\in \calC$. Since $\bbP = \bbQ_\emptyset$ and $\bbQ = \bbQ_{\calC}$, it then suffices to show that 
\begin{align*}
    \bbQ_U[\calB^{\Xi}]\leq \bbQ_{U\cup \{(x,m)\}}[\calB^{\Xi}].
\end{align*}
where, without loss of generality we can assume that $(x,m)\notin U$. For this purpose we will reproduce the proof of \cite[Theorem 2.1]{baumler2026localcriteriaglobalconnectivity}. Define the sigma algebra $\mathcal F^c_{(x,m)}$ as the smallest sigma algebra such that the random field $(N(y,i))_{(y,i)\neq (x,m)}$ is measurable. Conditionally on $\mathcal{F}^c_{(x,m)}$, there are two cases:
\begin{enumerate}
\item $(x,m)$ is {\em not pivotal} for $\calB^{\Xi}$. That is, if $N(x,m) = \emptyset$, then $\calB^{\Xi}$ still occurs, or, even if $N(x,m) =x+ \NN$, then $\calB^{\Xi}$ does not occur.
\item $(x,m)$ is {\em pivotal} for $\calB^{\Xi}$. That is, 
    if $N(x,m) = \emptyset$, then $\calB^{\Xi}$ does not occur, however, 
    there is some $a\in x + \NN$ such that, if $a\in N(x,m)$, then $\calB^{\Xi}$ occurs. 
\end{enumerate}
In the first case, there is nothing we have to do, because then the event $\calB^{\Xi}$ does not depend on $N(x,m)$.
In the second case, we observe that, conditionally on $\mathcal{F}^c_{(x,m)}$, there is some subset $\emptyset\subsetneq A\subseteq \NN$ such that $\calB^{\Xi}$ occurs if there is an open edge between $(x,m)$ and $(x+a,m+1)$ for $a\in A$. In this case, using our assumption on the ordering of the local laws,
\begin{align*}
    \bbQ_U[\calB^{\Xi}\cond\mathcal F^c_{(x,m)}] &= \bbQ_U[N'(x,m)\cap A\neq \emptyset\cond\mathcal F^c_{(x,m)}] = \bfP[N\cap A\neq \emptyset]\\&\leq \bfQ[N\cap A\neq \emptyset] = \bbQ_{U\cup \{(x,m)\}}[N'(x,m)\cap A\neq \emptyset\cond\mathcal F^c_{(x,m)}].
\end{align*}
The claim follows by taking the expectations. 

Finally, note that $\Xi$ is at most countable, so we can write $\Xi = \{\bar x^{(1)},\bar x^{(2)},\ldots\}$ for some enumeration. In particular,  $\Xi_k\uparrow\Xi$ for finite $\Xi_k: = \{\bar x^{(1)},\ldots,\bar x^{(k)}\}$ and the result for general $\Xi$ follows from continuity of measures. 
\end{proof}

\begin{proof}[Proof of Proposition~\ref{prop:exch}]
The statement can be proved using the same arguments as in the proof of~\cite[Proposition~3.2]{baumler2026localcriteriaglobalconnectivity}. For the convenience of the reader we reproduce the lower bound here. The proof for the upper bound can be performed by equally closely mimicking the proof of ~\cite[Proposition~3.2]{baumler2026localcriteriaglobalconnectivity}. 

Let $\emptyset\subsetneq A\subseteq \NN$ and $x\in A$. Then, for any exchangeable measure $\bfP$ of $N$ with $\bfP[\abs{N}] = (2d+1)p$ we have
\begin{align*}
    \bfP[N\cap A\neq \emptyset]&\geq \bfP[x\in N ] = \sum_{n = 1}^{2d+1}\bfP\big[x\in N\cond \abs{N} = n\big]\bfP[\abs{N} = n]\\ &= \sum_{n = 1}^{2d+1}\frac{n}{2d+1}\bfP[\abs{N} = n] = \frac{1}{2d+1}\bfP[\abs{N}] = p  = \bfP_p^{\text{aon}}[N\cap A\neq \emptyset],
\end{align*}
as desired.
\end{proof}
\begin{proof}[Proof of Lemma~\ref{lemma:burst}]

For $A\subseteq \NN$, we define
\begin{align*}
\hat{\bfP}[N=A]= 
\begin{cases} 1-q &\text{ if }\{o\}=A,  \\ 
0 &\text{ if }\{o\}\neq A,\ o\in A, \\
\bfP[N\setminus\{o\}=A] &\text{ if }A\neq \emptyset,o\not\in A,\\
q - \bfP[\abs{N\setminus\{o\}}>0]&\text{ if }\emptyset = A.
\end{cases}
\end{align*} 
To begin, we wish to show that $\hat{\bfP}$ is a valid probability distribution. By our assumptions on $\bfP,$ $0\leq \hat{\bfP}[N = A]\leq 1$ for all $A\subseteq \NN$. Similarly,
\begin{align*}
   \hat{\bfP}[o\not\in N]&=q-\bfP[\abs{N\setminus\{o\}}>0]+\sum_{k = 1}^{2d}\bfP[\abs{N\setminus\{o\}} = k]=q, \text{ and }\\
   \sum_{A\subseteq \NN} \hat{\bfP}[N=A] &= \hat{\bfP}[o\in N] + \hat{\bfP}[N = \emptyset] +\sum_{k = 1}^{2d}\bfP[\abs{N\setminus\{o\}} = k] = 1.
\end{align*}
Additionally,
\begin{align*}
    \bfP[\abs{N}] = 1- q + \sum_{k = 1}^{2d}k\bfP[\abs{N\setminus\{o\}} = k] = 1-q + \sum_{k = 1}^{2d}k\hat{\bfP}[\abs{N\setminus\{o\}} = k] =  \hat{\bfP}[\abs{N}].
\end{align*}
Finally, we want to show that $\bfP[N\cap A \neq \emptyset]\leq \hat{\bfP}[N\cap A\neq \emptyset]$ for all $\emptyset\subsetneq A\subseteq \NN$. If $\{o\} = A$, then
\begin{align*}
    \bfP[N\cap A\neq\emptyset ] = \hat{\bfP}[N\cap A\neq \emptyset] = 1-q.
\end{align*}
If $o\notin A$, then
\begin{align*}
    \bfP[N\cap A\neq \emptyset] = \bfP[(N\setminus\{o\})\cap A\neq \emptyset] = \hat{\bfP}[N\cap A\neq \emptyset].
\end{align*}
 In the final case, where $o\in A$ and $\{o\}\neq A$, we have that
 \begin{align*}
    \bfP[N\cap A\neq \emptyset]&= \bfP[o\in N]+q\bfP[(N\setminus \{o\})\cap A\neq \emptyset]=1-q+q\bfP[(N\setminus \{o\})\cap A\neq \emptyset],\\
    \hat{\bfP}[N\cap A\neq \emptyset] &= \hat{\bfP}[o\in N]+\bfP[(N\setminus \{o\})\cap A\neq \emptyset]=1-q+\bfP[(N\setminus \{o\})\cap A\neq \emptyset],
\end{align*}
as desired.
\end{proof}

\subsection{Continuous time}
We continue with the proofs for the continuous-time setting.

\begin{proof}[Proof of Theorem~\ref{thm_cont}]
First, assume that $\Xi(I,R)$ is such that for almost all $(I,R)$ we have that $\sup_{\bar x\in\Xi(I,R)}\abs{\bar x}  \leq  n$, where $\abs{\bar x}$ stands for the number of steps in $\bar x$. Then, for almost all $(I,R)$ the set of relevant space-time points, i.e., 
\begin{align*}
   \calC =  \calC(\Xi(I,R),(I,R)) :=\bigcup_{\substack{\bar{x}\in \Xi:\, \abs{\bar{x}}\leq n }} &\{(x,t)\in \bar{x}\},
\end{align*}
is finite. Now we can use the same arguments as in the proof of Theorem~\ref{thm_disc} to interpolate between $\bbP$ and $\bbQ$ by first conditioning on $(I,R)$ and then step-by-step replacing the probabilities on the relevant space-time positions $(x,t)\in \calC$.
This gives
\begin{align*}
    \bbP[\calB^{\Xi}]\le \bbQ[\calB^{\Xi}].
\end{align*}
We conclude the proof by writing $\Xi_n$ for the subset of $\Xi$ consisting of paths of lengths less than $n+1$, we have that $\calB^{\Xi_n}\uparrow\calB^{\Xi}$ for almost all $(I,R)$ and thus the result follows by an application of the monotone-convergence theorem with respect to $n\uparrow\infty$.
\end{proof}
\begin{proof}[Proof of Proposition~\ref{prop:cts_ord1}]
    We first restrict our attention to the case where $a = i, b =j,$ with $i,j$ positive integers. 
    We then have to show that
\begin{align*}
    \hat{\bfQ}_{j}[N\cap A\neq \emptyset]\leq \hat{\bfQ}_{i}[N\cap A\neq \emptyset],\qquad\text{for all }\emptyset \subsetneq A\subseteq \calN_o.
\end{align*}
For this, let $\ell = \abs{A}$ and first assume $\ell\leq 2d - j$. Then, we have to show that
\begin{align*}
    \frac{1}{i}\left(1 - \frac{\binom{2d - \ell}{i}}{\binom{2d}{i}}\right)\geq \frac{1}{j}\left(1 - \frac{\binom{2d - \ell}{j}}{\binom{2d}{j}}\right).
\end{align*}
Let $m:= 2d$ and for $k\in \{0, 1,\ldots, 2d-\ell\}$ write
\begin{align*}
    r(k) = \frac{\binom{m-\ell}{k}}{\binom{m}{k}} = \prod_{n = 0}^{\ell - 1}\frac{m - k -n}{m-n} = \prod_{n =0}^{\ell - 1}\left(1 - \frac{k}{m-n}\right).
\end{align*}
Using the product on the right-hand side, $r(k)$ can be defined for real $k\in [0,m-\ell ]$. We now want to show that $r(k)$ is convex on this interval. Write $h(k) = \ln(r(k)) = \sum_{n=0}^{\ell-1}\ln(m-k -n) + c_{\ell,m}$, where $c_{\ell,m}$ is a constant depending only on $\ell$ and $m$. Then,
\begin{align*}
    r''(k) = r(k)\left[\left(\sum_{n = 0}^{\ell -1}\frac{1}{m-k -n}\right)^2 - \sum_{n = 0}^{\ell - 1}\frac{1}{(m - k - n)^2}\right]\geq 0,
\end{align*} 
which implies that $r(k)$ is convex. To see the non-negativity, note that the second summand simply removes the diagonal term in the first summand.
We now observe that $g(k):= 1 -r(k)$ is concave, with $g(0) = 0$, so
\begin{align*}
    \hat{\bfQ}_k[N\cap A\neq \emptyset] = \frac{g(k) - g(0)}{k - 0},
\end{align*}
is the slope of the chord from the origin, which is non-increasing in $k$ for a concave function.

Next, we treat the case where $\ell\geq 2d - j + 1$. Based on the previous result, without loss of generality we can assume $\ell = 2d - j+1$ and $i = 2d - \ell = j-1$. All we need to show in this case is that $\binom{2d}{j-1}\geq j$, which can be deduced from the fact that
\begin{align*}
    \binom{2d}{j-1}\geq \binom{2d- 1}{j-1}\geq \ldots \geq \binom{j}{j - 1} = j.
\end{align*}
Finally, notice that for non-integer $a,b$ we may use that
\begin{align*}
    \varepsilon\mapsto \frac{(1-\varepsilon)}{k}g(k) + \frac{\varepsilon}{k+1}g(k+1)
\end{align*}
is non-increasing for $\varepsilon\in[0,1]$ since it interpolates the inequality $g(k)/k\ge g(k+1)/(k+1)$.
\end{proof}

\begin{proof}[Proof of Corollary~\ref{cts_crit}]
The first inequality follows by combining Theorem~\ref{thm_cont} and the comment after Proposition~\ref{prop:exch} for exchangeable measures on $\calP(\calN_o)$.

The second inequality can be deduced by considering a thinning of the family of i.i.d.~Poisson point processes $I$ determining the infection events. We obtain this new i.i.d.~family $I'=(I'_x)_{x\in\Z^d}$ by removing all space-time positions $(x,t)\in I_x$ for which $N(x,t) = \emptyset$. If the processes in $I$ have rate $2d\lambda$, then the processes in $I'$ have rate $2d\bfP[\abs{N}>0]\lambda$ and for each $(x,t)\in I'_x$, $N(x,t)$ contains at least one uniformly chosen neighbour of $x$. Taking $\lambda>\lambda_c/\bfP[\abs{N}>0]$ then concludes the proof.
\end{proof}

\subsection*{Acknowledgements.} We would like to thank Lukas Lüchtrath for helpful discussions. BJ and PS further acknowledge the financial support of the Leibniz Association within the Leibniz Junior Research Group on \emph{Probabilistic Methods for Dynamic Communication Networks} as part of the Leibniz Competition and by Deutsche Forschungsgemeinschaft (DFG, German Research Foundation) under Germany's Excellence Strategy -- The Berlin Mathematics Research Center MATH+ (EXC-2046/2, project ID: 390685689) through the project {\em EF-MA-Sys-2} on {\em Information Flow \& Emergent Behavior in Complex Networks}.

\section*{References}
\renewcommand*{\bibfont}{\footnotesize}
\printbibliography[heading = none]

@misc{baumler2026localcriteriaglobalconnectivity,
      title={Local criteria for global connectivity comparisons: beyond stochastic domination}, 
      author={Johannes Bäumler and Benedikt Jahnel and Jonas Köppl and Bas Lodewijks and Lily Reeves and András Tóbiás},
      year={2025},
      eprint={2510.03934},
      archivePrefix={arXiv},
      primaryClass={math.PR},
      url={https://arxiv.org/abs/2510.03934}, 
}

@article{williams2024reproduction,
  title={The reproduction number and its probability distribution for stochastic viral dynamics},
  author={Williams, Bevelynn and Carruthers, Jonathan and Gillard, Joseph J and Lythe, Grant and Perelson, Alan S and Ribeiro, Ruy M and Molina-Par{\'\i}s, Carmen and L{\'o}pez-Garc{\'\i}a, Mart{\'\i}n},
  journal={Journal of the Royal Society Interface},
  volume={21},
  number={210},
  pages={20230400},
  year={2024}
}

@article{Harris1974,
  author = {Theodore E. Harris},
  title = {Contact interactions on a lattice},
  journal = {Annals of Probability},
  year = {1974},
  volume = {2},
  number = {6},
  pages = {969--988}
}

@book{Liggett1985,
  author = {Thomas M. Liggett},
  title = {Interacting Particle Systems},
  publisher = {Springer},
  year = {1985}
}

@article{durrett1984oriented,
  title={Oriented percolation in two dimensions},
  author={Durrett, Richard},
  journal={The Annals of Probability},
  pages={999--1040},
  year={1984},
  publisher={JSTOR},
  volume = {12},
  number = {4}
}

@article{allard2023role,
  title={The role of directionality, heterogeneity, and correlations in epidemic risk and spread},
  author={Allard, Antoine and Moore, Cristopher and Scarpino, Samuel V and Althouse, Benjamin M and H{\'e}bert-Dufresne, Laurent},
  journal={SIAM Review},
  volume={65},
  number={2},
  pages={471--492},
  year={2023},
  publisher={SIAM}
}

@article{grimmett1998dependent,
  title={Dependent random graphs and spatial epidemics},
  author={van den Berg, J and Grimmett, Geoffrey R and Schinazi, Rinaldo B},
  journal={The Annals of Applied Probability},
  volume={8},
  number={2},
  pages={317--336},
  year={1998},
  publisher={Institute of Mathematical Statistics}
}

@article{schinazi2000horizontal,
  title={Horizontal versus vertical transmission of parasites in a stochastic spatial model},
  author={Schinazi, Rinaldo B},
  journal={Mathematical Biosciences},
  volume={168},
  number={1},
  pages={1--8},
  year={2000},
  publisher={Elsevier}
}

@article{miller2008bounding,
  title={Bounding the size and probability of epidemics on networks},
  author={Miller, Joel C},
  journal={Journal of Applied Probability},
  volume={45},
  number={2},
  pages={498--512},
  year={2008},
  publisher={Cambridge University Press}
}

@article{meester2011bounding,
  title={Bounding basic characteristics of spatial epidemics with a new percolation model},
  author={Meester, Ronald and Trapman, Pieter},
  journal={Advances in Applied Probability},
  volume={43},
  number={2},
  pages={335--347},
  year={2011},
  publisher={Cambridge University Press}
}

@article{cox1988limit,
  title={Limit theorems for the spread of epidemics and forest fires},
  author={Cox, J T  and Durrett, R.},
  journal={Stochastic Processes and Their Applications},
  volume={30},
  number={2},
  pages={171--191},
  year={1988},
  publisher={North-Holland}
}

@article{mcdiarmid1981general,
  title={General percolation and random graphs},
  author={McDiarmid, C.},
  journal={Advances in Applied Probability},
  volume={13},
  number={1},
  pages={40--60},
  year={1981},
  publisher={Cambridge University Press}
}

@book{Liggett1999,
  author = {Thomas M. Liggett},
  title = {Stochastic Interacting Systems: Contact, Voter and Exclusion Processes},
  publisher = {Springer},
  year = {1999}
}

@article{PearsonKrapivskyPerelson2011,
  author = {John E. Pearson and Pavel L. Krapivsky and Alan S. Perelson},
  title = {Stochastic theory of early viral infection: Continuous versus burst production of virions},
  journal = {PLoS Computational Biology},
  volume = {7},
  number = {2},
  year = {2011},
  pages = {e1001058}
}

@article{Kuulasmaa1982,
 ISSN = {00219002},
 URL = {http://www.jstor.org/stable/3213827},
 author = {Kari Kuulasmaa},
 journal = {Journal of Applied Probability},
 number = {4},
 pages = {745--758},
 publisher = {Applied Probability Trust},
 title = {The spatial general epidemic and locally dependent random graphs},
 urldate = {2026-08-17},
 volume = {19},
 year = {1982}
}

@article{Kuulasmaa_Zachary_1984,
 title={On spatial general epidemics and bond percolation processes},
 volume={21},
% DOI={10.2307/3213706},
 number={4},
 journal={Journal of Applied Probability},
 author={Kuulasmaa, Kari and Zachary, Stan},
 year={1984},
 pages={911–914}
 }
\end{document}